\documentclass[twocolumn,nonote,noversion,squeezed]{cdcarticle}

\def\MODE{3}

\usepackage[utf8]{inputenc}				% handle accents properly
\usepackage[english]{babel}				% handle hyphenation properly

\usepackage{math}						% custom math macros
\usepackage{enumitem}
\usepackage{cite}
\usepackage[font=small,labelfont=bf]{caption}

\usepackage[hidelinks]{hyperref}
\usepackage{subcaption}

\usepackage{tikz,pgfplots}
\usetikzlibrary{external,calc,positioning}
\pgfplotsset{compat=1.14}
\graphicspath{{./graphics/}}
\usepgfplotslibrary{external,colorbrewer}
\newcommand{\df}{\nabla\!f}
\newcommand{\dg}{\nabla\!g}

\let\oldbibliography\thebibliography
\renewcommand{\thebibliography}[1]{\oldbibliography{#1}
\setlength{\itemsep}{1pt}} %Reducing spacing in the bibliography.

\begin{document}

\title{A Robust Accelerated Optimization Algorithm\\for Strongly Convex Functions}

\if\MODE1
\author{Saman Cyrus$^{1,2}$, Bin Hu$^{2}$, Bryan Van Scoy$^{2}$, Laurent Lessard$^{1,2}$ % <-this % stops a space
\thanks{This work was supported by National Science Foundation (CISE/CCF). Award Number 1656951.}% <-this % stops a space
\thanks{$^{1}$ The authors are with the Department of Electrical and Computer Engineering,
        University of Wisconsin--Madison, Madison, WI 53706, USA.
        }%
\thanks{$^{2}$ The authors are with the Optimization Group at the Wisconsin Institute for Discovery.
        Email: {\tt\small \{cyrus2,bhu38, vanscoy,laurent.lessard\}@wisc.edu}}%
}

\markboth{IEEE Transactions on Automatic Control}%
{}
\fi
\if\MODE4
\author{Saman~Cyrus\footnotemark[1]$^,$\footnotemark[2] \and
        Bin~Hu\footnotemark[2] \and 
        Bryan~Van~Scoy\footnotemark[2] \and
        Laurent~Lessard\footnotemark[1]$^,$\footnotemark[2]}
\note{}
\fi
\if\MODE3
\author{Saman~Cyrus \and
	Bin~Hu \and 
	Bryan~Van~Scoy \and
	Laurent~Lessard}
\note{}
\fi

\maketitle

\if\MODE4
\footnotetext[1]{S.~Cyrus and L.~Lessard are with the Department of Electrical and Computer Engineering, University of Wisconsin--Madison.}
\footnotetext[2]{All authors are with the Wisconsin Institute for Discovery at the University of Wisconsin--Madison.
	Email addresses:\\ {\small\texttt{\{cyrus2,bhu38,vanscoy,laurent.lessard\}@wisc.edu}}}
\footnotetext[3]{This material is based upon work supported by the National Science Foundation under Grants No. 1656951 and 1750162.}
\fi

%%%%%%%%%%%%%%%%%%%%%%%%%%%%%%%%%%%%%%%%%%%%%%%%%%%%%%%%%%%%%%%%%%%%%%%%%%%%%%%%
\begin{abstract}
This work proposes an accelerated first-order algorithm we call the Robust Momentum Method for optimizing smooth strongly convex functions. The algorithm has a single scalar parameter that can be tuned to trade off robustness to gradient noise versus worst-case convergence rate. At one extreme, the algorithm is faster than Nesterov's Fast Gradient Method by a constant factor but more fragile to noise. At the other extreme, the algorithm reduces to the Gradient Method and is very robust to noise. The algorithm design technique is inspired by methods from classical control theory and the resulting algorithm has a simple analytical form. Algorithm performance is verified on a series of numerical simulations in both noise-free and relative gradient noise cases.
\end{abstract}

%%%%%%%%%%%%%%%%%%%%%%%%%%%%%%%%%%%%%%%%%%%%%%%%%%%%%%%%%%%%%%%%%%%%%%%%%%%%%%%%
\section{Introduction}

Consider the unconstrained optimization problem
\begin{equation}
\min_{x\in \mathbb{R}^n} f(x) \label{eq:optproblem}
\end{equation}
where $f:\R^n\rightarrow \R$ is $L$-smooth and $m$-strongly convex. The strong convexity of $f$ guarantees that there exists a unique minimizer $x_\star$ satisfying $\df(x_\star)=0$. First-order methods are widely used for solving~\eqref{eq:optproblem} when the Hessian is prohibitively expensive to compute, e.g., when the problem dimension is large.
A simple first-order algorithm for solving~\eqref{eq:optproblem} is the Gradient Method (GM),
\begin{equation*}
x_{k+1} = x_k - \alpha \df(x_k), \qquad x_0 \in \R^n.
\end{equation*}
For smooth and strongly convex $f$, the GM with a well-chosen stepsize converges linearly to the optimizer \cite{polyak1987introduction}. That is, for some $c \ge 0$ and $\rho \in [0,1)$, we have
\begin{equation*}
\| x_k - x_* \| \le c\, \rho^k \quad\text{for all }k\geq 0.
\end{equation*}
For example, the standard choice $\alpha=1/L$ leads to a linear rate $\rho=1-\frac{m}{L}$, while the choice $\alpha=\frac{2}{L+m}$ results in the improved linear rate $\rho=\frac{L-m}{L+m}$. 

The issue with the Gradient Method, however, is that the convergence rate is slow, especially for ill-conditioned problems where the ratio $\frac{L}{m}$ is large. A common method of accelerating convergence is to use \emph{momentum}. A well-established momentum algorithm for smooth and strongly convex $f$ is Nesterov's Fast Gradient Method\footnote{Also called Neterov's accelerated gradient method.}, (FGM)~\cite{nesterov2004introductory} described by the iteration
\begin{align*}
x_{k+1} &= y_k - \alpha \df(y_k), \qquad\quad x_0,x_{-1}\in\mathbb{R}^n \\
y_k &= x_k + \beta(x_k-x_{k-1}).
\end{align*}
The FGM tuned with $\alpha=\tfrac{1}{L}$ and $\beta = \tfrac{\sqrt{L}-\sqrt{m}}{\sqrt{L}+\sqrt{m}}$ converges with rate $\rho^2 < 1-\sqrt{m/L}$, which is faster than the GM rate\footnote{A numerical study in~\cite{lessard2016analysis} revealed that the standard rate bound for FGM derived in~\cite{nesterov2004introductory} is conservative. Nevertheless, the bound has a simple algebraic form and is asymptotically tight.}. The rate can be improved to $\rho = 1-\sqrt{m/L}$ using an accelerated algorithm called the Triple Momentum Method~\cite{van2018fastest}. This is the fastest known worst-case convergence rate for this class of problems.

Robustness issues arise naturally in many optimization problems. For example, achieving the above rates associated with each first-order method requires knowledge of $L$ and $m$, which may not be accurately accessible in practice. In addition, the gradient evaluation can be inexact for certain applications \cite{d2008smooth, schmidt2011, devolder2014first}. These issues motivate the need for accelerated first-order methods that are robust to underlying design assumptions. 

As observed in \cite[\S 5.2]{lessard2016analysis}, optimization algorithm design involves a tradeoff between performance and robustness. For example, consider stepsize tuning for the GM. Using $\alpha=\frac{2}{L+m}$ optimizes the convergence rate, but makes the algorithm fragile to gradient noise. The more conservative choice $\alpha=\frac{1}{L}$ results in slower convergence, but more robustness to noise. This is consistent with the intuition that a smaller stepsize can improve the algorithm's robustness at the price of degrading its performance. For momentum methods, exploiting the tradeoff between performance and robustness is less straightforward, since one has to tune multiple algorithm parameters in a coupled manner to achieve acceleration. This tradeoff is exploited in \cite{devolder2013intermediate} for first-order methods applied to smooth convex problems. In this work, we design a first-order method that exploits the tradeoff between robustness and performance for smooth strongly convex problems.

\paragraph{Notation.}
The set of functions that are $m$-strongly convex and $L$-smooth is denoted $\mathcal{F}(m,L)$. In particular, $f\in\mathcal{F}(m,L)$ if for all $x,y \in \mathbb{R}^n$,
\begin{equation*}
m \| x-y \|^2 \le \left( \df(x) - \df(y) \right)^\tp  (x-y) \le L \| x-y \|^2.
\end{equation*}
The condition ratio is defined as $\kappa \defeq L/m$.

%%%%%%%%%%%%%%%%%%%%%%%%%%%%%%%%%%%%%%%%%%%%%%%%%%%%%%%%%%%%%%%%%%%%%%%%%%%%%%%%%
\section{Main result}\label{sec:main}

\subsection{Robust Momentum Method}\label{sec:rmm}
Our proposed algorithm is parameterized by a scalar $\rho$ that represents the worst-case convergence rate of the algorithm in the noise-free case. Specifically, the iteration is governed by the following recursion with arbitrary initialization $x_0,x_{-1}\in\mathbb{R}^n$
\begin{subequations}\label{algo}
\begin{align}
x_{k+1} &= x_k + \beta (x_k - x_{k-1}) - \alpha \df(y_k), \label{algo1}\\
y_k &= x_k + \gamma (x_k - x_{k-1}).
\end{align}
\end{subequations}
where $\alpha$, $\beta$, and $\gamma$ depend directly on the parameter $\rho$ as
\begin{equation}\label{eq:params}
  \begin{gathered}
  \alpha = \frac{\kappa (1-\rho)^2 (1+\rho)}{L},\qquad
  \beta = \frac{\kappa  \rho ^3}{\kappa -1}, \\
  \gamma = \frac{\rho ^3}{(\kappa -1) (1-\rho)^2 (1+\rho)}.
  \end{gathered}
\end{equation}
We now state the key convergence property of the Robust Momentum Method in the noise-free case.

\begin{thm}
\label{thm:main}
Suppose $f\in\mathcal{F}(m,L)$ with $0<m\leq L$ and let $x_\star$ be the unique minimizer of $f$. Given the  parameter $\rho \in [1-1/\sqrt{\kappa},\,1-1/\kappa]$, the Robust Momentum Method~\eqref{algo} with parameter tuning~\eqref{eq:params} satisfies the bound
\begin{align}
\norm{x_k-x_\star} &\le c\,\rho^{k}\qquad\text{for }k\ge 1 \label{eq:bound_x}
\end{align}
where $c > 0$ is a constant that does not depend on $k$.
\end{thm}
The proof of Theorem~\ref{thm:main} is provided in Section~\ref{sec:MainProof}.
Theorem~\ref{thm:main} states that $\rho$ directly controls the worst-case convergence rate of the Robust Momentum Method. We will see in Section~\ref{sec:controlinter} that although increasing $\rho$ makes the algorithm slower, it also makes it more robust to gradient noise. In particular,
\begin{itemize}
\item The minimum value is $\rho = 1-1/\sqrt{\kappa}$. This is the fastest achievable convergence rate and also leads to the most fragile algorithm. This choice recovers the Triple Momentum Method~\cite{van2018fastest}.
\item The maximum value is $\rho = 1-1/\kappa$. This is the slowest achievable convergence rate and also leads to the most robust algorithm. This choice recovers the Gradient Method with stepsize $\alpha=1/L$.

\end{itemize}
To see why this last case reduces to the Gradient Method, substitute $\rho = 1-1/\kappa$ into~\eqref{algo} and~\eqref{eq:params}. Then,~\eqref{algo1} reduces to $y_{k+1} = y_k - \tfrac{1}{L}\df(y_k)$.

%%%%%%%%%%%%%%%%%%%%%%%%%%%%%%%%%%%%%%%%%%%%%%%%%%%%%%%%%%%%%%%%%%%%%%%%%%%%%%%%

\subsection{Convergence rate proof}
\label{sec:MainProof}

In this section, we derive a proof for Theorem \ref{thm:main}. The approach that follows is similar to the one used in~\cite{lessard2016analysis}, with one important difference. In addition to proving a rate bound as in~\cite{lessard2016analysis}, we also derive
a Lyapunov function that yields intuition for the algorithm's behavior and robustness properties.

\begin{prop}[Co-coercivity] \label{prop:co-coercivity}
Suppose $f:\R^n\to \R$ is convex and differentiable. Further suppose $f$ is $L$-smooth. Then for all $x,y \in \R^n$,
\[
f(y) \ge f(x) + \df(x)^\tp (y-x) + \frac{1}{2L}\norm{\df(y) - \df(x)}^2.
\]
\end{prop}

The following lemma proves a key property of strongly convex functions. Parts of this result appear  in~\cite{lessard2016analysis} and we repeat them here for completeness.
\begin{lem}\label{lem:sequences}
Suppose $f\in\mathcal{F}(m,L)$. Let $x_\star$ be the unique minimizer of $f$ (i.e., $\df(x_\star)=0$). Define the function $g(x) \defeq f(x) - f(x_\star) - \frac{m}{2}\norm{x-x_\star}^2$. Given any sequence of points $\{y_k\} \subseteq \R^n$, 
\begin{enumerate}
\item If we define $q_k \defeq (L-m)g(y_k) - \tfrac{1}{2}\norm{\dg(y_k)}^2$, then
\[
q_k \ge 0\quad\text{for all }k.
\]
\item If we define $u_k \defeq \df(y_k)$ and $\tilde y_k \defeq y_k - x_\star$, then
\[
(u_k - m \tilde y_k )^\tp (L\tilde y_k - u_k) \ge q_k\quad\text{for all }k.
\]
\item Using the same definitions as above, the following inequality holds for any $0 \le \rho \le 1$,
\begin{multline*}
( u_k - m \tilde y_k )^\tp \bigl( L(\tilde y_k-\rho^2 \tilde y_{k-1}) \\
  - (u_k-\rho^2 u_{k-1}) \bigr) \ge q_k - \rho^2 q_{k-1} \quad\text{for all }k.
\end{multline*}
\end{enumerate}
\end{lem}
\begin{proof}
By the definition of strong convexity, $g$ is convex and $(L-m)$-smooth. Moreover, $g(y) \ge g(x_\star) = 0$ for all $y\in\R^n$. Item~1 follows from applying Proposition~\ref{prop:co-coercivity} with $(f,x,y) \mapsto (g,x_\star,y_k)$. For Item~2, note that $u_k = \df(y_k) = \dg(y_k)+m\tilde y_k$. We have
\begin{align*}
(u_k-m\tilde y_k)^\tp (L\tilde y_k-u_k)
&=\dg(y_k)^\tp\bigl((L\!-\!m) \tilde y_k \!-\! \dg(y_k)\bigr) \\
&\ge (L-m) g(y_k) - \tfrac{1}{2}\|\dg(y_k) \|^2 \\
&= q_k
\end{align*}
where the inequality follows from applying Proposition~\ref{prop:co-coercivity} with $(f,x,y)\mapsto (g,y_k,x_\star)$. To prove Item~3, begin with the case $\rho=1$. Using a similar argument to the one used to prove Item~2,
\begin{align*}
&(u_k - m\tilde y_k)^\tp \bigl( L(\tilde y_k-\tilde y_{k-1}) - (u_k-u_{k-1}) \bigr) \\
&=\dg(y_k)^\tp \bigl((L\!-\!m) (\tilde y_k-\tilde y_{k-1}) - (\dg(y_k)-\dg(y_{k-1}) )\bigr) \\
&\ge q_k - q_{k-1}
\end{align*}
where the inequality follows from applying Proposition~\ref{prop:co-coercivity} with $(f,x,y)\mapsto (g,y_k,y_{k-1})$. By combining the two previous results, we have
\begin{align*}
&(u_k - m \tilde y_k )^\tp \bigl( L(\tilde y_k-\rho^2 \tilde y_{k-1}) - (u_k-\rho^2 u_{k-1}) \bigr) \\
&= (1-\rho^2)(u_k - m \tilde y_k )^\tp \bigl( L\tilde y_k - u_k) \\
&\qquad\quad + \rho^2 (u_k - m \tilde y_k )^\tp \bigl( L(\tilde y_k -\tilde y_{k-1}) - (u_k-u_{k-1}) \bigr) \\
&\ge (1-\rho^2)q_k + \rho^2(q_k - q_{k-1}) \\
&= q_k - \rho^2 q_{k-1}
\end{align*}
and this completes the proof of Item~3.
\end{proof}

Our next lemma provides a key algebraic property of the Robust Momentum Method~\eqref{algo}. This result makes no assumptions about $f$.

\begin{lem}\label{lem:identity}
Suppose $\{u_k,x_k,y_k\}$ is any sequence of vectors satisfying the constraints
\begin{equation}\label{eq:constr1}
  \bmat{x_{k+1} \\ y_k } =
  \bmat{ 1+\beta  & -\beta  & -\alpha \\
         1+\gamma & -\gamma & 0 }
  \bmat{ x_k \\ x_{k-1} \\ u_k }
  \quad\text{for }k\ge 0
\end{equation}
where $(\alpha,\beta,\gamma)$ are given by~\eqref{eq:params}, and thus depend on the parameters $0<m\le L$, $\kappa \defeq L/m$, and $\rho\in(0,1)$. Define
$z_k \defeq (1-\rho^2)^{-1}\left(x_k - \rho^2 x_{k-1}\right)$ for $k \ge 0$.
Then the following algebraic identity holds for $k\ge 1$,
\begin{multline}\label{eq:dissipation}
  (u_k - m y_k )^\tp \bigl( L(y_k-\rho^2 y_{k-1}) - (u_k-\rho^2 u_{k-1}) \bigr) \\
  + \lambda \left( \norm{z_{k+1}}^2 - \rho^2 \norm{z_k}^2\right) + \nu \norm{ u_k - m y_k}^2 = 0
\end{multline}
where the constants $\lambda$ and $\nu$ are defined as
\begin{align}
  \lambda &\defeq \frac{m^2 \left(\kappa-\kappa  \rho ^2 -1\right)}{2\rho  (1-\rho )}
  \quad\text{and} \label{eq:deflambda}\\
  \nu &\defeq \frac{(1+\rho) \left(1-\kappa+2 \kappa  \rho -\kappa  \rho ^2\right)}{2\rho }. \label{eq:defc}
\end{align}
\end{lem}

\begin{proof}
The algebraic identity may be verified by direct substitution of~\eqref{eq:params}, \eqref{eq:constr1}, \eqref{eq:deflambda}, and \eqref{eq:defc} into~\eqref{eq:dissipation}. Specifically, the constraints~\eqref{eq:constr1} allow us to express $z_{k+1}$, $z_k$, $y_k$, $y_{k-1}$, $u_k$, and $u_{k-1}$ as linear functions of $x_k$, $x_{k-1}$, $x_{k-2}$, and $u_k$. Upon doing so, the resulting expression becomes identically zero. To express $u_{k-1}$ as required, rearrange the first equation of~\eqref{eq:constr1} to obtain the expression $u_{k-1}={\alpha}^{-1}((1+\beta)x_{k-1}-\beta x_{k-2}-x_k)$.
\end{proof}

The algebraic identity~\eqref{eq:dissipation} has three main terms. We will see how each serves a role in explaining the convergence and robustness properties of our algorithm. We are now ready to prove Theorem \ref{thm:main}.

\paragraph{Proof of Theorem \ref{thm:main}.}
Choose $x_0$ and $x_{-1}$ arbitrarily and consider the sequence $\{u_k,x_k,y_k,z_k\}$ defined by setting $u_k \defeq \df(y_k)$ and propagating for all $k\ge 0$ using~\eqref{eq:constr1}. This sequence is precisely a trajectory of our algorithm. Let $x_\star$ be the unique minimizer of $f$. Define the shifted sequences $\tilde x_k \defeq x_k - x_\star$, $\tilde y_k \defeq y_k - x_\star$, and $\tilde z_k \defeq z_k - x_\star$ where $z_k$ is defined in Lemma~\ref{lem:identity}. Note that the constraints~\eqref{eq:constr1} still hold when we use the shifted sequence $\{u_k,\tilde x_k,\tilde y_k,\tilde z_k\}$. Applying Lemma~\ref{lem:identity}  with Item~3 of Lemma~\ref{lem:sequences}, we conclude that for $k\ge 1$,
\begin{multline}\label{eq:diss1}
\lambda (\norm{\tilde z_{k+1}}^2 - \rho^2 \norm{\tilde z_k}^2) + (q_k - \rho^2 q_{k-1}) \\
+ \nu\, \norm{ u_k - m \tilde y_k}^2 \le 0,
\end{multline}
where $\lambda$ and $\nu$ are defined in \eqref{eq:deflambda}--\eqref{eq:defc}.
When $1-1/\sqrt{\kappa} \le \rho \le 1-1/\kappa$, we have $mL \ge \lambda \ge \tfrac{1}{2}mL$ and $0 \le \nu \le 1-\tfrac{1}{2\kappa}$. As we increase $\rho$, the parameter $\lambda$ decreases monotonically while $\nu$ increases monotonically.
Define the sequence $\{V_k\}$ by $V_k \defeq \lambda \norm{\tilde z_k}^2 + q_{k-1}$. If we choose $\rho$ in the interval specified above, then $\nu\ge 0$ and $\lambda > 0$. Since $q_k\ge 0$, $V_k$ can serve as a Lyapunov function. In particular, it follows from~\eqref{eq:diss1} that
\begin{equation}\label{lyap}
V_{k+1} \le \rho^2\, V_k\qquad\text{for }k\ge 1.
\end{equation}
Iterating this relationship, we find that $V_{k+1} \le \rho^{2k}\, V_1$. The reason we do not iterate down to zero is because $V_k$ is not defined at $k=0$. Substituting the definitions and simplifying, we obtain the bound
\begin{equation}\label{eq:bound_z}
\norm{\tilde z_{k+1}} \le \rho^{k} \sqrt{ \norm{\tilde z_1}^2 + \lambda^{-1}q_0}\qquad\text{for }k\ge 1.
\end{equation}
The bound~\eqref{eq:bound_z} therefore captures two effects. As we increase~$\rho$, the linear rate~$\rho^k$ becomes slower and the constant factor in the rate bound also grows.

Next, we show that $\{\tilde x_k\}$ goes to zero at the same rate~$\rho^k$, but with different constant factors. Note that because $\tilde z_k =(1-\rho^2)^{-1}\left(\tilde x_k - \rho^2 \tilde x_{k-1}\right)$, we can form the telescoping sum
\begin{equation}\label{telescope}
\tilde x_k = \rho^{2(k-1)} \tilde x_{-1} + (1-\rho^2) \sum_{t=0}^{k-1} \rho^{2(k-t)} \tilde z_t
\quad\text{for }k\ge 0.
\end{equation}
Taking the norm of both sides of~\eqref{telescope}, applying the triangle inequality, and substituting~\eqref{eq:bound_z}, we obtain a geometric series. Upon simplification, we find that $\norm{\tilde x_k}$ is bounded above by a constant times $\rho^k$, as required.

%%%%%%%%%%%%%%%%%%%%%%%%%%%%%%%%%%%%%%%%%%%%%%%%%%%%%%%%%%%%%%%%%%%%%%%%%%%%%%%%
\section{Control design interpretations}\label{sec:controlinter}

In this section, we cast the problem of algorithm analysis as a robust control problem. Specifically, we can view the problem of algorithm analysis as being equivalent to solving a Lur'e problem~\cite{lur1944theory}. The Lur'e setup is illustrated in Figure~\ref{fig:Interconnection}, where a linear dynamical system $G$~\eqref{eq:optalg} is in feedback with a static nonlinearity $\phi$.
\begin{figure}[htb]
	\centering
	\begin{minipage}{0.4\linewidth}
		\centering
		\tikzsetnextfilename{block_diagram}
		\begin{tikzpicture}[thick]
		\tikzstyle{block}=[draw,rectangle,inner sep=2mm,minimum width=0.9cm,minimum height=0.8cm]
		\tikzstyle{link}=[->,>=latex]
		\node (G)    [block] at (0,0) {$G$};
		\node (phi)  [block] at (0,-1) {$\phi$};
		\draw [link] (G.east) -- +(.6,0) |- node[auto,pos=0.25] {$y$} (phi);
		\draw [link] (phi.west) -- +(-.6,0) |- node[auto,pos=0.25] {$u\vphantom{y}$} (G);
		\end{tikzpicture}
	\end{minipage}%
	\begin{minipage}{0.6\linewidth}
		\begin{subequations} \label{eq:optalg}
			\begin{align}
			\xi_{k+1} &= A \xi_k + B u_k, \\
			y_k     &= C  \xi_k, \\
			u_k     &= \phi( y_k).
			\end{align}
		\end{subequations}
	\end{minipage}
	\caption{Feedback interconnection of a linear system $G$ with a troublesome (nonlinear or uncertain) component $\phi$. We use the positive feedback convention in this block diagram.}
	\label{fig:Interconnection}
\end{figure}
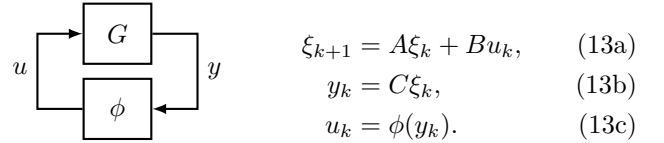

The Robust Momentum Method (as well as the Fast Gradient Method and ordinary Gradient Method) can be written in this way by setting $\phi = \df$ and choosing $A$, $B$, and $C$ appropriately. For example, the Robust Momentum Method~\eqref{algo} is given by
\begin{align*}
A &= \bmat{1+\beta & -\beta\\ 1 & 0}, &
B &= \bmat{-\alpha \\ 0 }, &
C &= \bmat{1 + \gamma & -\gamma}.
\end{align*}
Here, we shifted all signals so they are measured relative to the steady-state value $x_\star$ and therefore assumed that $\df(0)=0$. We also assumed without loss of generality that $u_k$ and $y_k$ are scalars. This interpretation was used in~\cite{lessard2016analysis,hu17a,hu17b} to provide a unified analysis framework.

Traditionally, Lur'e systems were analyzed in the frequency domain rather than the time domain. For the case of the Robust Momentum Method, the (discrete-time) transfer function of the linear block is given by
\begin{equation} \label{eq:RMMXfun}
G(z) = -\alpha \frac{(1 + \gamma) z - \gamma}{(z-1)(z-\beta)}.
\end{equation}
It was observed in Section~\ref{sec:rmm} that the Robust Momentum Method becomes the Gradient Method if $\rho = 1-1/\kappa$. This fact can be directly verified using the transfer function.
Substituting this $\rho$ and the parameter values~\eqref{eq:params} into~\eqref{eq:RMMXfun}, there is a pole-zero cancellation and we obtain $G(z) = \frac{-1}{L(z-1)}$, which is the transfer function for the Gradient Method with stepsize $\alpha=\frac{1}{L}$.

\paragraph{Frequency-domain condition.}
Continuing with the frequency-domain interpretation, Lur'e systems can be analyzed using the formalism of Integral Quadratic Constraints (IQCs)~\cite{megretski1997system}. To this end, the nonlinearity is characterized by a quadratic inequality that holds between its input and output
\[
\int_{|z|=1} \bmat{\hat y(z) \\ \hat u(z)}^* \Pi(z) \bmat{\hat y(z) \\ \hat u(z)}\mathrm{d}z \ge 0
\]
where $\hat y$ and $\hat u$ are the $z$-transforms of $\{y_k\}$ and $\{u_k\}$, respectively, and $\Pi(z)$ is a para-Hermitian matrix. For convenience, we use a loop-shifting transformation to move the nonlinearity $\phi = \df$ from the sector $(m,L)$ to the sector $(0,\kappa-1)$. We also scale the frequency variable $z$ by a factor of $\rho$ so that we can reduce the problem of certifying exponential stability (finding a linear rate) to that of certifying BIBO stability. This procedure is described in~\cite{boczar2015exponential}.

The nonlinearity of interest is sector-bounded and slope-restricted because it is the gradient of a function $g\in\mathcal{F}(0,\kappa-1)$. We may therefore represent the nonlinearity with a Zames--Falb IQC as in~\cite{boczar2015exponential}, leading to
\[
\Pi(z) \defeq 
\bmat{ 0 & (\kappa-1)(1-\rho^{2}\bar{z}^{-1})\\
(\kappa-1)(1-\rho^{2}z^{-1}) & -2 + \rho^{2}(z^{-1}+\bar{z}^{-1})}
.\]
The transformed transfer function is
\begin{equation} \label{eq:RMMXfun2}
\tilde G(z) = \frac{-\alpha m (1 + \gamma) z + \alpha m \gamma}{z^2 - (1+\beta -\alpha m (1+\gamma)) z + \beta -\alpha m \gamma}.
\end{equation}
To certify stability of the feedback interconnection, we must have $\tilde G(\rho z)$ stable and for all  $|z| = 1$,
\begin{equation}\label{eq:TM}
\Re\left( (1-\rho z^{-1})\bigl((\kappa-1)\tilde G(\rho z)-1\bigr) \right)  < 0.
\end{equation}
Equation \eqref{eq:TM} has a graphical interpretation; that the Nyquist plot of
$F(z)\defeq (1-\rho z^{-1})\bigl((\kappa-1)\tilde G(\rho z)-1\bigr)$ should lie entirely in the left half-plane. 

\paragraph{Graphical design for robustness.}
The frequency-domain condition \eqref{eq:TM} can provide useful intuition for the design of robust accelerated optimization methods. We can visualize different algorithms by choosing the parameters $\alpha,\beta,\gamma$ appropriately in~\eqref{eq:RMMXfun2}.

In Figure~\ref{Fig:nyquist_plots} (left panel), we show the Nyquist plot for the Gradient Method using the sector IQC~\cite{boczar2015exponential,lessard2016analysis}. To this effect, we set $\beta=\gamma=0$ and use either $\alpha=\tfrac{2}{L+m}$ or $\alpha=\tfrac{1}{L}$. As we increase $\rho$, the Nyquist plots become ellipses in the left half-plane. At the fastest certifiable rate (smallest $\rho$), the plots become vertical lines. When $\alpha=\tfrac{2}{L+m}$, the vertical line coincides with the imaginary axis, whereas when $\alpha=\tfrac{1}{L}$, the vertical line is shifted left. This result confirms our intuition that since the imaginary axis is the stability boundary, \textit{robust} stability is achieved as the Nyquist contour moves further left, away from the boundary.

The Robust Momentum Method~\eqref{algo} was designed such that the Nyquist diagram forms a vertical line passing through the point $(-\nu,0)$. In other words, we solved for $(\alpha, \beta, \gamma)$ such that~\eqref{eq:TM} holds with the right-hand side replaced by $-\nu$. Constraining the Nyquist plot as such directly leads to the choice~\eqref{eq:params} with $\nu$ related to $\rho$ via~\eqref{eq:defc}.
In Figure~\ref{Fig:nyquist_plots} (right panel), we show the Nyquist plot for the Robust Momentum Method using the Zames--Falb IQC (for $\nu=0$ and $\nu=\tfrac{1}{2}$). We also show Nyquist plots that certify a convergence rate of $\rho$ that is larger than the corresponding algorithm parameter. This leads to ellipses as with the Gradient Method. Note that although the RMM and GM plots look similar, the RMM $\rho$-values are generally smaller due to acceleration.
In contrast, the FGM (center panel) does not produce a vertical line in the Nyquist plot but still touches the stability boundary at the optimal $\rho$.

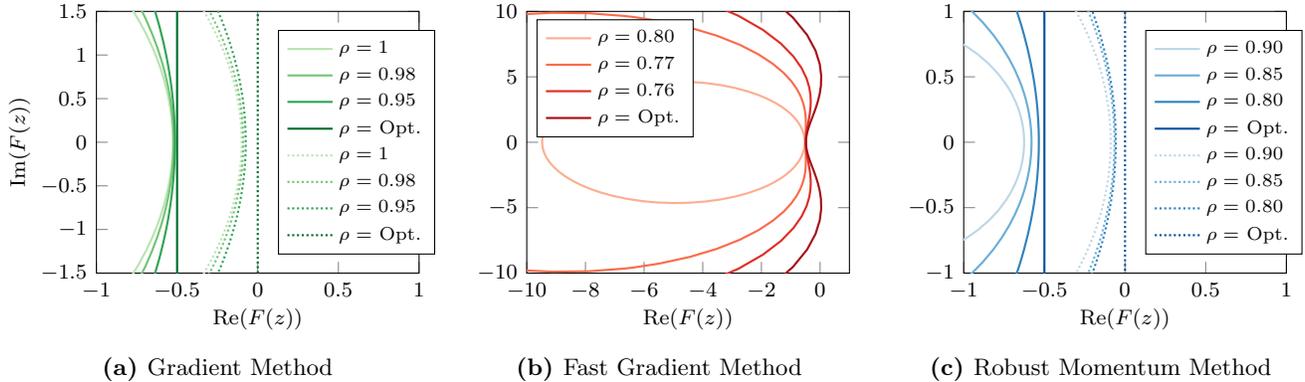
\begin{figure*}[htb]
	\begin{subfigure}{0.33\textwidth}
		\centering
		\tikzsetnextfilename{Nyquist_curves1}
		\begin{tikzpicture}
		\pgfplotsset{
			cycle list/Greens,
			cycle multiindex* list={
				mark=none\nextlist
				thick\nextlist
				Greens-E,Greens-G,Greens-I,Greens-K\nextlist
				solid,solid,solid,solid,densely dotted,densely dotted,densely dotted,densely dotted\nextlist}
		}
		\begin{axis}[
		width=\textwidth,
		y post scale=1,
		xmin =-1,xmax= 1,
		ymin =-1.5,ymax=1.5,
		xlabel={$\Re(F(z))$},
		ylabel={$\Im(F(z))$},
		xlabel shift = -1mm,
		label style={font=\footnotesize},
		xtick={-1,-0.5,0,0.5,1},
		ytick={-1.5,-1,-0.5,0,0.5,1,1.5},
		ylabel shift= -1mm,
		legend cell align=left,
		legend style={font=\scriptsize},
		legend style={at={(1.05,0.5)},anchor=east},
		ticklabel style={font=\footnotesize},
		legend entries={$\rho = 1$,$\rho = 0.98$,$\rho = 0.95$,$\rho =$ Opt.,$\rho = 1$,$\rho = 0.98$,$\rho = 0.95$,$\rho =$ Opt.}
		]
		\addplot table [x index=0,y index=1,header=false] {data/NyGD1.dat};
		\addplot table [x index=0,y index=1,header=false] {data/NyGD2.dat};
		\addplot table [x index=0,y index=1,header=false] {data/NyGD3.dat};
		\addplot table [x index=0,y index=1,header=false] {data/NyGD4.dat};	
		\addplot table [x index=0,y index=1,header=false] {data/NyGDa1.dat}; 
		\addplot table [x index=0,y index=1,header=false] {data/NyGDa2.dat}; 
		\addplot table [x index=0,y index=1,header=false] {data/NyGDa3.dat}; 
		\addplot table [x index=0,y index=1,header=false] {data/NyGDa4.dat};
		\end{axis}
		\end{tikzpicture}
		\caption{Gradient Method}
	\end{subfigure}%
	\begin{subfigure}{0.33\textwidth}
		\centering
		\tikzsetnextfilename{Nyquist_curves2}
		\begin{tikzpicture}
		\pgfplotsset{
			cycle list/Reds,
			cycle multiindex* list={
				mark=none\nextlist
				thick\nextlist
				Reds-E,Reds-G,Reds-I,Reds-K\nextlist
				solid\nextlist}
		}
		\begin{axis}[
		width=\textwidth,
		y post scale=1,
		xmin=-10,xmax= 1,
		ymin=-10,ymax=10,
		xlabel={$\Re(F(z))$},
		xlabel shift = -1mm,
		ylabel shift = -4mm,
		xtick={-10,-8,-6,-4,-2,0},
		ytick={-10,-5,0,5,10},
		legend cell align=left,
		legend style={font=\scriptsize},
		label style={font=\footnotesize},
		legend pos=north west,
		ticklabel style={font=\footnotesize},
		legend entries={$\rho = 0.80$,$\rho = 0.77$,$\rho = 0.76$,$\rho =$ Opt.}
		]
		\addplot table [x index=0,y index=1,header=false] {data/NyNest1.dat};
		\addplot table [x index=0,y index=1,header=false] {data/NyNest2.dat};
		\addplot table [x index=0,y index=1,header=false] {data/NyNest3.dat};
		\addplot table [x index=0,y index=1,header=false] {data/NyNest4.dat};
		\end{axis}
		\end{tikzpicture}
		\caption{Fast Gradient Method}
	\end{subfigure}%
	\begin{subfigure}{0.33\textwidth}
		\centering
		\tikzsetnextfilename{Nyquist_curves3}
		\begin{tikzpicture}
		\pgfplotsset{
			cycle list/Blues,
			cycle multiindex* list={
				mark=none\nextlist
				thick\nextlist
				Blues-E,Blues-G,Blues-I,Blues-K\nextlist
				solid,solid,solid,solid,densely dotted,densely dotted,densely dotted,densely dotted\nextlist}
		}
		\begin{axis}[
		width=\textwidth,
		y post scale=1,
		xmin=-1,xmax= 1,
		ymin=-1,ymax=1,
		xlabel={$\Re(F(z))$},
		xlabel shift = -1mm,
		ylabel shift = -3mm,
		xtick={-1,-0.5,0,0.5,1},
		ytick={-1,-0.5,0,0.5,1},
		legend cell align=left,
		legend style={font=\scriptsize},
		label style={font=\footnotesize},
		legend style={at={(1.05,0.5)},anchor=east},
		ticklabel style={font=\footnotesize},
		legend entries={$\rho = 0.90$,$\rho = 0.85$,$\rho = 0.80$,$\rho =$ Opt.,$\rho = 0.90$,$\rho = 0.85$,$\rho = 0.80$,$\rho =$ Opt.}
		]
		\addplot table [x index=0,y index=1,header=false] {data/NyRMMa1.dat};
		\addplot table [x index=0,y index=1,header=false]{data/NyRMMa2.dat};
		\addplot table [x index=0,y index=1,header=false]{data/NyRMMa3.dat};
		\addplot table [x index=0,y index=1,header=false]{data/NyRMMa4.dat};
		\addplot table [x index=0,y index=1,header=false] {data/NyRMM1.dat}; 
		\addplot table [x index=0,y index=1,header=false]{data/NyRMM2.dat}; 
		\addplot table [x index=0,y index=1,header=false]{data/NyRMM3.dat}; 
		\addplot table [x index=0,y index=1,header=false]{data/NyRMM4.dat};
		\end{axis}
		\end{tikzpicture}
		\caption{Robust Momentum Method}
	\end{subfigure}
	\caption{Frequency-domain plots of various algorithms for $\kappa=10$ and different values of the convergence rate $\rho$. The system is stable if the entire curve lies in the left half-plane. \textbf{(a)}~Gradient Method for $\alpha=1/L$ (solid) and $\alpha=2/(L+m)$ (dashed). The latter is right on the stability boundary while the former is shifted left (more robust). \textbf{(b)}~Fast Gradient Method. \textbf{(c)}~Robust Momentum Method for $\nu=1/2$ (solid) and $\nu=0$ (dashed). Again, the latter is right on the stability boundary while the former is shifted left (more robust).}
	\label{Fig:nyquist_plots}
\end{figure*}

\paragraph{Further robustness interpretations.}
The parameter $\nu$ can be interpreted as the \textit{input feed-forward passivity index} (IFP)~\cite{bao2007process}, which is a measure of the shortage or excess of passivity of the system $F(z)$ defined above. In the frequency domain, the discrete-time definition of the IFP index is given by\footnote{Most sources use a negative feedback convention. The definition we give in~\eqref{eq:IFPfre} uses the positive feedback convention.}
\begin{equation}\label{eq:IFPfre}
\nu(F(z)) \defeq-\tfrac{1}{2} \max_{|z|=1}\, \lambda_\textup{max} \bigl(F(z)+F(z)^*\bigr),
\end{equation}
where $\lambda_\textup{max}(\cdot)$ denotes the largest eigenvalue and $F^*$ is the conjugate transpose of $F$. 
For the SISO case, \eqref{eq:IFPfre} reduces to $\nu=-\max_{|z|=1}\Re(F(z))$, which is the shortest distance between each curve and the imaginary axis in Figure~\ref{Fig:nyquist_plots}.

We can also interpret $\nu$ as a robustness margin in the time domain using the Lyapunov function defined in~\eqref{eq:defc}. In the proof of Theorem~\ref{thm:main}, when we substitute the definition for $V_k$ into~\eqref{eq:diss1}, we obtain
\[
V_{k+1} \le \rho^2\, V_k - \nu\,\norm{\dg(y_k)}^2.
\]
Proving the desired rate bound only requires~\eqref{lyap} to hold, so the term $\nu\,\norm{\dg(y_k)}^2$ can be interpreted as an additional margin that ensures the inequality $V_{k+1}\le \rho^2 V_k$ will hold even if underlying assumptions such as exactness in gradient evaluations or accurate knowledge of $L$ and $m$ are violated. As we increase $\rho$, the linear rate becomes slower, but $\nu$ also increases via~\eqref{eq:defc}, which serves to increase the robustness margin in the inequality~\eqref{lyap}.

%%%%%%%%%%%%%%%%%%%%%%%%%%%%%%%%%%%%%%%%%%%%%%%%%%%%%%%%%%%%%

\section{Robustness to gradient noise}\label{sec:simulations}
The Robust Momentum Method has a single parameter, which can be used to tune the performance. In this section, we provide both simulations and numerical rate analyses to verify the performance of the algorithm when the gradient is subject to relative deterministic noise~\cite{polyak1987introduction}. Specifically, we will suppose that instead of measuring the gradient $\df(y_k)$, we measure $u_k = \df(y_k) + r_k$ where $r_k\in\R^n$ satisfies $\norm{r_k} \leq \delta\,\norm{\df(y_k)}$. For a given fixed $\delta \ge 0$, we will bound the worst-case performance of the algorithm over all  $f\in\mathcal{F}(m,L)$ and feasible $\{r_k\}$.

\paragraph{Numerical rate analysis.}
To find the worst-case performance, we adopt the methodology from~\cite[Eq. 5.1]{lessard2016analysis}. There, the authors formulate a linear matrix inequality parameterized by $\hat \rho$ and $\delta$ whose feasibility provides a sufficient condition for convergence with linear rate~$\hat \rho$. 

In Figure~\ref{Fig:rho_delta}, we plot the computed convergence rate as a function of noise strength $\delta$ for the Gradient Method, Fast Gradient Method, and Robust Momentum Method.
Note that the worst-case rate in closed form for the Gradient Method is given in \cite{de2017onworst,2017arXiv170905191D}.

\begin{figure}[htb]	
	\tikzsetnextfilename{tradeoff_curves}
	\begin{tikzpicture}
	\begin{axis}[
	width=0.96\linewidth,
	y post scale=1.1,
	xmin=0,xmax=1,
	ymin=0.683772,ymax=1,
	xlabel={Noise strength ($\delta$)},
	ylabel={Upper bound on worst-case rate ($\hat \rho$)},
	ylabel shift = -2mm,
	ytick={0.683772,0.818182,0.9,1},
	yticklabels={$1-\dfrac{1}{\sqrt{\kappa}}$,$\dfrac{\kappa-1}{\kappa+1}$,$1-\dfrac{1}{\kappa}$,1},
	legend cell align=left,
	legend style={font=\scriptsize},
	label style={font=\footnotesize},
	legend pos=south east,
	ticklabel style={font=\footnotesize}]
	\addplot [Set1-C, thick, densely dashed] table [x index=0,y index=1,header=false] {data/rho_delta_kappa10.dat};
	\addlegendentry{GM ($\alpha=1/L$)}
	\addplot [Set1-C, thick, densely dotted] table [x index=0,y index=2,header=false] {data/rho_delta_kappa10.dat};
	\addlegendentry{GM ($\alpha=2/(L+m)$)}
	\addplot [Set1-C, very thick, solid] table [x index=0,y index=3,header=false] {data/rho_delta_kappa10.dat};
	\addlegendentry{GM (min $\alpha\in[0,2/L]$)}
	\addplot [Set1-A, very thick, solid] table [x index=0,y index=4,header=false] {data/rho_delta_kappa10.dat};
	\addlegendentry{FGM}
	\addplot [Set1-B, thick, densely dashed] table [x index=0,y index=5,header=false] {data/rho_delta_kappa10.dat};
	\addlegendentry{RMM ($\nu=0$)}
	\addplot [Set1-B, thick, dash dot] table [x index=0,y index=6,header=false] {data/rho_delta_kappa10.dat};
	\addlegendentry{RMM ($\nu=1/3$)}
	\addplot [Set1-B, thick, densely dotted] table [x index=0,y index=7,header=false] {data/rho_delta_kappa10.dat};
	\addlegendentry{RMM ($\nu=2/3$)}
	\addplot [Set1-B, very thick, solid] table [x index=0,y index=9,header=false] {data/rho_delta_kappa10.dat};
	\addlegendentry{RMM (min $\nu\in[0,1)$)}
	\end{axis}
	\end{tikzpicture}
	\caption{Upper bound on the worst-case linear convergence rate as a function of the noise level $\delta$ for $\kappa=10$ (the figure looks similar for other choices of $\kappa$). We used a relative noise model, where the measured gradient $u_k$ satisfies $\norm{u_k-\df(y_k)}\leq\delta\,\norm{\df(y_k)}$ for the Gradient Method (GM), Fast Gradient Method (FGM), and Robust Momentum Method (RMM). By tuning the parameter $\nu$, the RMM trades off robustness to gradient noise with convergence rate.}
	\label{Fig:rho_delta}
\end{figure}
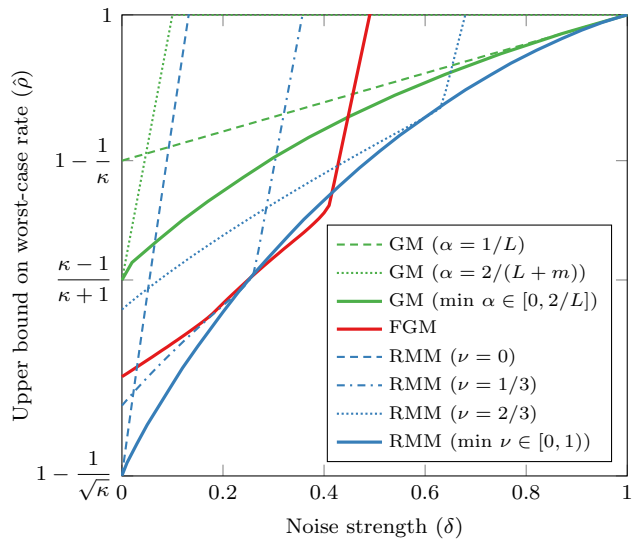

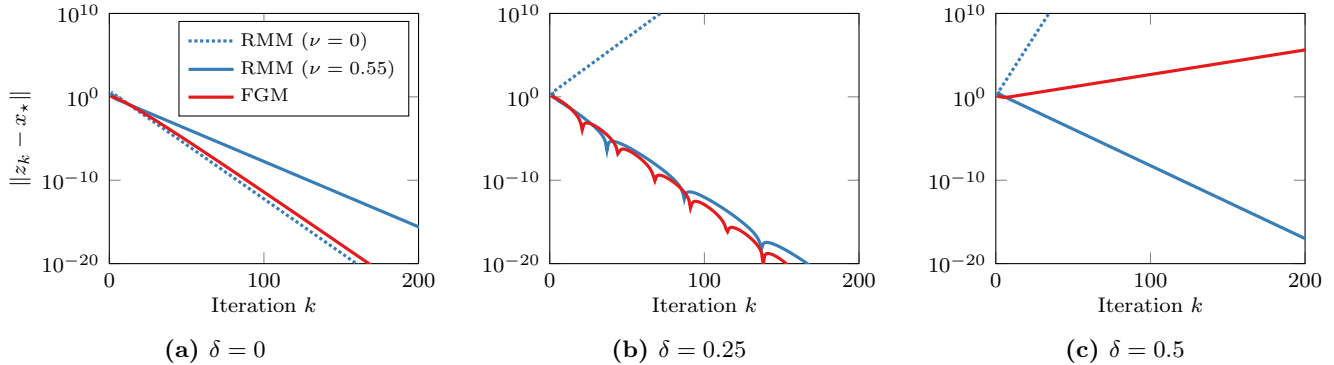
\begin{figure*}[htb]
	\begin{subfigure}{0.32\textwidth}
		\centering
		\tikzsetnextfilename{algo_simulation1}
		\begin{tikzpicture}
		\begin{semilogyaxis}[
		width=\textwidth,
		y post scale=1,
		xmin=0,xmax=200,
		ymin=1e-20,ymax=1e10,
		xlabel={Iteration $k$},
		ylabel={$\|z_k - x_\star\|$},
		ylabel shift = -1mm,
		xlabel shift = -1mm,
		xtick={0,100,200,300},
		ytick={1e-20,1e-10,1e0,1e10},
		legend cell align=left,
		legend style={font=\scriptsize},
		legend pos=north east,
		ticklabel style={font=\footnotesize},
		label style={font=\footnotesize},
		solid]
		\addplot [Set1-B, very thick, densely dotted] table [x index=0,y index=1,header=false] {data/Simulation_noisy_delta1.dat}; \addlegendentry{RMM ($\nu = 0$)}
		\addplot [Set1-B, very thick] table [x index=0,y index=2,header=false] {data/Simulation_noisy_delta1.dat}; \addlegendentry{RMM ($\nu = 0.55$)}
		\addplot [Set1-A, very thick] table [x index=0,y index=3,header=false] {data/Simulation_noisy_delta1.dat}; \addlegendentry{FGM}
		\end{semilogyaxis}
		\end{tikzpicture}
		\caption{$\delta = 0$}
	\end{subfigure}\hfill\hfill%
	\begin{subfigure}{0.32\textwidth}
		\centering
		\tikzsetnextfilename{algo_simulation2}
		\begin{tikzpicture}
		\begin{semilogyaxis}[
		width=\textwidth,
		y post scale=1,
		xmin=0,xmax=200,
		ymin=1e-20,ymax=1e10,
		xlabel={Iteration $k$},
		xlabel shift = -1mm,
		xtick={0,100,200,300},
		ytick={1e-20,1e-10,1e0,1e10},
		legend cell align=left,
		legend pos=north east,
		ticklabel style={font=\footnotesize},
		label style={font=\footnotesize},
		solid]
		\addplot [Set1-B, very thick, densely dotted] table [x index=0,y index=1,header=false] {data/Simulation_noisy_delta2.dat}; 
		\addplot [Set1-B, very thick] table [x index=0,y index=2,header=false] {data/Simulation_noisy_delta2.dat};
		\addplot [Set1-A, very thick] table [x index=0,y index=3,header=false] {data/Simulation_noisy_delta2.dat};
		\end{semilogyaxis}
		\end{tikzpicture}
		\caption{$\delta = 0.25$}
	\end{subfigure}\hfill%
	\begin{subfigure}{0.32\textwidth}
		\centering
		\tikzsetnextfilename{algo_simulation3}
		\begin{tikzpicture}
		\begin{semilogyaxis}[
		width=\textwidth,
		y post scale=1,
		xmin=0,xmax=200,
		ymin=1e-20,ymax=1e10,
		xlabel={Iteration $k$},
		xlabel shift = -1mm,
		xtick={0,100,200,300},
		ytick={1e-20,1e-10,1e0,1e10},
		legend cell align=left,
		legend pos=north east,
		ticklabel style={font=\footnotesize},
		label style={font=\footnotesize},
		solid]
		\addplot [Set1-B, very thick, densely dotted] table [x index=0,y index=1,header=false] {data/Simulation_noisy_delta3.dat};
		\addplot [Set1-B, very thick] table [x index=0,y index=2,header=false] {data/Simulation_noisy_delta3.dat};
		\addplot [Set1-A, very thick] table [x index=0,y index=3,header=false] {data/Simulation_noisy_delta3.dat};
		\end{semilogyaxis}
		\end{tikzpicture}
		\caption{$\delta = 0.5$}
	\end{subfigure}
	\caption{Simulation of the Robust Momentum Method (RMM) and the Fast Gradient Method (FGM) with relative gradient noise of strength $\delta$ and condition ratio $\kappa=10$. The objective function is the two-dimensional quadratic with gradient~\eqref{Eq:gradf}. The measured gradient at each iteration is $u_k = (1-\delta)\df(y_k)$. \textbf{(a)}~With no noise, all methods are stable and the RMM with $\nu=0$ is the fastest. \textbf{(b)}~With more noise, the RMM with $\nu=0$, the most fragile possible tuning, is unstable. \textbf{(c)}~With high noise, only the RMM with $\nu=0.55$ remains stable. Even FGM is unstable with this much noise.}
	\label{Fig:Simulation}
\end{figure*}

First, consider the Robust Momentum Method. When $\nu=0$ and there is no gradient noise ($\delta=0$), the method achieves the fast convergence rate $1-1/\sqrt{\kappa}$. Increasing the noise level above $\delta>0.13$, however, leads to a loss of convergence guarantee. As we increase $\nu$, the convergence rate becomes slower but the method is capable of tolerating larger noise levels. In the limiting case as $\nu= 1-\frac{1}{2\kappa}$ the Robust Momentum Method becomes the Gradient Method with $\alpha=\frac{1}{L}$ (dashed black line).

It is interesting to note that the Fast Gradient Method has a faster convergence bound than the Robust Momentum Method for noise levels $0.26<\delta<0.41$. However, the Fast Gradient Method is also unstable for $\delta>0.5$ while the Robust Momentum Method can be tuned so that it converges with noise levels up to $\delta\to 1$.

\paragraph{Numerical simulations.}
To illustrate the noise robustness properties of different tunings of the Robust Momentum Method, we compared it to the Fast Gradient Method when applied to a simple two-dimensional quadratic function. We used the gradient
\begin{align} \label{Eq:gradf}
\df(y_k) &= \begin{bmatrix} m & 0 \\ 0 & L \end{bmatrix} (y_k-x_\star)
\end{align}
where the gradient noise is $r_k = -\delta\,\df(y_k)$.
See Figure~\ref{Fig:Simulation}.
The RMM with $\nu=0$ has the fastest convergence rate in the noiseless case ($\delta=0$), but quickly diverges when noise is present. The FGM is more robust to noise, but also diverges when the noise magnitude $\delta$ is too large. The RMM with $\nu=0.55$ remains stable for large amounts of noise, although in the absence of noise the convergence rate is slower than both other methods.

%%%%%%%%%%%%%%%%%%%%%%%%%%%%%%%%%%%%%%%%%%%%%%%%%%%%%%%%%%%%%%%%%%%%%%%%%%%%%%%%
\bibliographystyle{IEEEtran}
\begin{small}
\bibliography{references}
\end{small}

\end{document}